\documentclass[11pt]{article}

\usepackage{fullpage}
\usepackage{amsmath, amsthm, amsfonts, amssymb, amstext, mathrsfs, enumerate}
\usepackage{graphicx, ragged2e, lscape, framed, xcolor}
\usepackage{subfiles}
\usepackage{todonotes}

\theoremstyle{plain}
\newtheorem{theorem}{Theorem}[section]
\newtheorem{lemma}[theorem]{Lemma}

\newtheorem{conjecture}[theorem]{Conjecture}

\newtheorem{problem}[theorem]{Problem}
\newtheorem{claim}{Claim}[subsection]

\numberwithin{equation}{section}
\allowdisplaybreaks

\newcommand{\affl}[3]{\noindent #1, Email: {\tt #2}\\ \textsc{#3}\\[1.5pt]}

\usepackage[pagebackref]{hyperref}
\hypersetup{
	colorlinks=true,
    urlcolor=purple,
	linkcolor=purple,
    citecolor=purple,
}

\DeclareMathOperator{\Tr}{Tr}
\DeclareMathOperator{\supp}{supp}
\DeclareMathOperator{\N}{N}
\DeclareMathOperator{\Gal}{Gal}
\def\ol{\overline}
\def\q{\mbox{$q^{\#}$}}
\def\Q{\mathbb Q}
\def\K{\mathbb K}
\def\C{\mathbb C}

\title{\textbf{A proof of the cyclotomic conjecture and the non-existence of almost Moore digraphs}}
\author{Jaskaran Kaur \and Hitesh Kumar}
\date{}

\begin{document}
\maketitle
\begin{abstract}
For $n>2$ and $k>1$, define the polynomial
\[F_{n,k}(x) = \Phi_n(1 + x + \cdots + x^k),\]
where $\Phi_n$ denotes the $n$-th cyclotomic polynomial. The \emph{cyclotomic conjecture} proposed by Gimbert (1999) exactly describes the irreducibility of $F_{n,k}(x)$ over $\Q$ in terms of $n$ and $k$. Conde, Gimbert, Gonz\'{a}lez, Miller and Miret (2014) established that the cyclotomic conjecture, if true, would imply the non-existence of almost Moore digraphs - a well-known open question concerning the directed degree-diameter problem. In this article, we prove the cyclotomic conjecture and, as a consequence, show that there are no almost Moore digraphs with maximum out-degree $d$ and diameter $k$ for any $d>1$ and $k>2$.
\end{abstract}



\section{Introduction}

\subsection{Motivation and background}

Given natural numbers $\Delta$ and $D$, the \emph{degree-diameter problem} asks for the largest order $n_{\Delta, D}$ of a graph with maximum degree at most $\Delta$ and diameter at most $D$. This is an extensively studied problem, and several variants of it have been investigated. We refer to the excellent dynamic survey by Miller and \v{S}ir\'{a}\v{n} \cite{Miller_Siran_2005}. 

In this paper, we focus on the directed version. 

\begin{problem}[Directed degree-diameter problem] Given natural numbers $d$ and $k$, let $\mathcal{D}(d,k)$ denote the collection of digraphs with maximum out-degree at most $d$ and diameter at most $k$. Determine the largest order $n_{d, k}$ of a digraph in $\mathcal{D}(d,k)$.     
\end{problem}

Define 
\[ M_{d,k}:= 1 + d + \cdots +d^k.\]
It is known that
\begin{equation}\label{eq:Moore_bound}
  n_{d,k}\le M_{d,k}.  
\end{equation}
The quantity $M_{d,k}$ is called the (directed) \emph{Moore bound} and the digraphs in $\mathcal{D}(d,k)$ of order $M_{d,k}$ (if they exist) are called \emph{Moore digraphs}.

It is well-known that Moore digraphs are rare and exist only in the trivial cases $d = 1$ (directed cycles $C_{k+1}$) or $k=1$ (complete digraphs $K_{d+1}$); see \cite{Plesnik_Znam_1974, Bridges_Toueg_1980, Conde_Gimbert_Gonzalez_Miller_Miret_2014}.

Naturally, then one wonders if there exist digraphs in $\mathcal{D}(d,k)$ of order $M_{d,k}-1$. Such digraphs are called \emph{almost Moore digraphs} or simply $(d,k)$-\emph{digraphs}. 

Fiol, Alegre and Yebra \cite{Fiol_Alegre_Yebra_1983} proved that the directed line graph of the directed complete graph $K_{d+1}$ is a $(d,2)$-digraph. Gimbert \cite{Gimbert_1999, Gimbert_2001} completely classified the $(d,2)$-digraphs. Indeed, for $d\ge 3$ the $(d,2)$-digraphs are precisely the directed line graphs of directed complete graphs, and there are three non-isomorphic $(2,2)$-digraphs. Conde, Gimbert, Gonz\'{a}lez, Miret and Moreno \cite{Conde_Gimbert_Gonzalez_Miret_Moreno_2008, Conde_Gimbert_Gonzalez_Miret_Moreno_2013} proved the non-existence of $(d,3)$-digraphs and $(d,4)$-digraphs for $d > 1$. 

Miller and Fri\v{s} \cite{Miller_Fris_1993} proved that there are no $(2,k)$-digraphs for $k\ge 3$. Baskoro, Miller, \v{S}ir\'{a}\v{n} and Sutton \cite{Baskoro_Miller_Siran_Sutton_2005} (cf. \cite{Baskoro_Miller_Plesnik_Znam_1995}) established the non-existence of $(3,k)$-digraphs for $k\ge 3$. See \cite{Cholily_2011, Sillasen_2015, Conde_Miller_Miret_Kumar_2015, Lopez_Messegue_Miret_2023,Miret_Simanjuntak_Simon_2023, Baskoro_Messegue_Miret_2025,Messegue_Miret_Sillasen_2026} for other progress. The full conjecture remains wide open.

\subsection{Connection to cyclotomic polynomials}

For $n\in \N$, the $n$-th \emph{cyclotomic polynomial} is defined to be 
\[ \Phi_n(x):= \prod_{\substack{1\le j\le n\\ \gcd(j,n) = 1}} \left(x - e^{2i\pi \frac{j}{n}} \right).\]
In words, $\Phi_n(x)$ is the minimal polynomial over $\Q$ with integer coefficients for any primitive $n$-th root of unity $\zeta_n$. 

In 1999, Gimbert \cite{Gimbert_1999} proposed the following conjecture.

\begin{conjecture}[Cyclotomic conjecture \cite{Gimbert_1999}]\label{conj:cyclotomic} Let $n> 2$, $k>1$ and
\[ F_{n,k}(x) := \Phi_n(1 + x + \cdots + x^k).\]
\begin{enumerate}[$(i)$]
    \item If $k$ is even, then $F_{n,k}(x)$ is reducible in $\Q[x]$ if and only if $n|(k+2)$. In that case, $F_{n,k}(x)$ has precisely two irreducible factors.
    \item If $k$ is odd, then $F_{n,k}(x)$ is reducible in $\Q[x]$ if and only if $n$ is even and $n|2(k+2)$. In that case, $F_{n,k}(x)$ has precisely two irreducible factors.
\end{enumerate}
\end{conjecture}

 Gimbert \cite{Gimbert_1999} related the irreducibility of the polynomial $F_{n,k}(x)$ over $\Q[x]$ to the factorization of characteristic polynomials of $(d,k)$-digraphs (see \cite[Proposition 2]{Conde_Gimbert_Gonzalez_Miller_Miret_2014}). In 2014, Conde, Gimbert, Gonz\'{a}lez, Miller and Miret \cite{Conde_Gimbert_Gonzalez_Miller_Miret_2014} proved that the truth of Conjecture \ref{conj:cyclotomic} indeed implies the non-existence of almost Moore digraphs. 

\begin{theorem}[\cite{Conde_Gimbert_Gonzalez_Miller_Miret_2014}]\label{thm:conditional_non_existence}
Let $d>1$ and $k>2$. If the cyclotomic conjecture holds for every pair $(n,k)$, where $2<n\le d + \cdots +d^k$, then $(d,k)$-digraphs do not exist.  
\end{theorem}

Gimbert \cite{Gimbert_1999} proved that $F_{2,k}(x)$ is irreducible for all $k>1$. Lenstra Jr. and Poonen \cite{Lenstra_Poonen_1998} (cf. \cite{Gimbert_2001}) proved the cyclotomic conjecture for $k=2$. Conde, Gimbert, Gonz\'{a}lez, Miret and Moreno established the cases $k=3$ and $k=4$ in \cite{Conde_Gimbert_Gonzalez_Miret_Moreno_2008} and \cite{Conde_Gimbert_Gonzalez_Miret_Moreno_2013}, respectively. 

\subsection{Our contribution}

We settle Conjecture \ref{conj:cyclotomic} in full generality. 

\begin{theorem}\label{thm:cyclotomic}
Let $n > 2$, $k>1$ and
\[ F_{n,k}(x) := \Phi_n(1 + x + \cdots + x^k).\]
\begin{enumerate}[$(i)$]
    \item If $k$ is even, then $F_{n,k}(x)$ is reducible in $\Q[x]$ if and only if $n|(k+2)$. In that case, $F_{n,k}(x)$ has precisely two irreducible factors.
    \item If $k$ is odd, then $F_{n,k}(x)$ is reducible in $\Q[x]$ if and only if $n$ is even and $n|2(k+2)$. In that case, $F_{n,k}(x)$ has precisely two irreducible factors.
\end{enumerate}
\end{theorem}

The non-existence of almost Moore digraphs is immediate using Theorem \ref{thm:cyclotomic} and Theorem \ref{thm:conditional_non_existence}. 

\begin{theorem}\label{thm:no_almost_Moore_digraph} The $(d,k)$-digraphs do not exist for any $d>1$ and $k>2$. 
\end{theorem}

In other words, for all $G\in \mathcal{D}(d,k)$, we have 
\[ |V(G)|\le M_{d,k}-2,\]
whenever $d>1$ and $k>2$. The next natural question is whether there exist digraphs in $\mathcal{D}(d,k)$ of order $M_{d,k}-2$; see \cite{Miller_Siran_2005} for some progress on this problem. 

The remainder of the article is devoted to the proof of Theorem \ref{thm:cyclotomic}. As noted earlier, Theorem \ref{thm:cyclotomic} is known to be true for $k=2, 3, 4$ in the literature, but their methods do not generalize as they use specific properties and ad-hoc techniques that only seem to work for small $k$. For completeness, we will give a proof for all $k>1$.

\section{Proof of the cyclotomic conjecture}
\label{sec:cyclotomic_proof}

\subsection{Notational setup}
Throughout, assume that $n>2$ and $k>1$. Define 
\[ s_k(x) := 1 + x + \cdots + x^k.\]
Let $\zeta_n$ denote a primitive $n$-th root of unity. Let 
\[\K := \Q(\zeta_n)\] 
and $\mathcal{O}_{\K}$ denote its ring of integers. Clearly, 
\begin{equation}\label{eq:F_n_k_expansion}
  F_{n,k}(x) = \Phi_n(s_k(x)) = \prod_{\sigma\in \Gal(\K/\Q)}\left(s_k(x)-\sigma(\zeta_n)\right),  
\end{equation}
where the product is over distinct primitive $n$-th roots of unity. The degree of $F_{n,k}(x)$ is $k\varphi(n)$, where $\varphi$ is Euler's totient function. 

If a polynomial $f(x) = \sum_{i=0}^r f_i\, x^i\in \K[x]$ has degree $r$ and the constant term $f_0\neq 0$, then define the \emph{conjugate reciprocal} of $f$ to be the polynomial
\[f^{\#}(x) := x^r\, \ol{f}(x^{-1}) = \sum_{i=0}^r \ol{f_i}\,x^{r-i},\]
where $\ol{f}$ (resp. $\ol{f_i}$) denotes the conjugation in $\K[x]$ (resp. $\K$). If $f$ and $g$ are polynomials in $\K[x]$ with non-zero constant terms, then clearly
\begin{equation}
    (fg)^{\#} = f^{\#}g^{\#} \quad \text{and}\quad (f^{\#})^{\#} = f.
\end{equation}
We define the \emph{support} of $f$ to be the set of indices of its non-zero coefficients, i.e.,
\[ \supp(f):=\{0\le i\le r: f_i\neq 0\}.\]

For a polynomial $f(x) = \sum_{i=0}^{k+1} f_i\ x^i\in \K[x]$ of degree $k+1$, define its $\ell$-th \emph{autocorrelation coefficient} by 
\[A_\ell(f):=\sum_{j=0}^{k+1-\ell} f_{j+\ell}\ol{f_j}\qquad (0\le \ell \le k+1).\]
If $f^{\#}$ is well-defined, then $A_\ell(f)$ is indeed the coefficient of $x^{k+1+\ell}$ in $ff^{\#}$. Moreover, 
\[ A_\ell(f) = A_\ell(f^{\#}) \quad (0\le \ell\le k+1).\]
We will denote by $\mu(n)$ the M\"{o}bius function. The \emph{trace} and \emph{norm} of an element $\beta\in \K$ over $\Q$ will be denoted by $\Tr_{\K/\Q}(\beta)$ and $\N_{\K/\Q}(\beta)$, respectively. We will use standard notation, terminology and results from Field Theory; refer to \cite{Ribenboim_2001, Dummit_Foote_2004} for details.

The following inequality is a standard application of the AM-GM inequality (see \cite[Ch. 5, exercise 27]{Ribenboim_2001}).

\begin{lemma}\label{lemma:trace_inequality}
For every non-zero $\beta\in \mathcal{O}_{\K}$, 
    \[ \Tr_{\K/\Q}(\beta \ol{\beta})\ge \varphi(n).\] 
\end{lemma}

\subsection{Polynomials $q(x)$ and $r(x)$}
Define
\begin{equation}
     q(x) := s_k(x) - \zeta_n  \in \mathcal{O}_{\K}[x].
\end{equation}
The degree of $q$ is $k$. Note here that $q(0) = 1 - \zeta_n\neq  0$, and so $q^{\#}$ is well-defined. Indeed
\begin{equation}
  \q(x) = s_k(x) - \zeta_n^{-1}x^k.   
\end{equation}
Equation \eqref{eq:F_n_k_expansion} suggests that the irreducibility of $F_{n,k}(x)$ over $\Q[x]$ is closely related to the irreducibility of $q(x)$ over $\K[x]$. We first analyze the polynomial $q(x)$, characterize precisely when $q(x)$ is irreducible, and then use it to prove the cyclotomic conjecture. 

We first determine the greatest common divisor of $q$ and $\q$ in $\K[x]$. 

\begin{lemma}\label{lemma:gcd} For $n> 2$ and $k>1$, we have
\[ \gcd(q, \q) = 
\begin{cases}
    1 & \text{ if }\, (-\zeta_n^{-1})^{k+2} \neq 1;\\
    x + \zeta_n^{-1} & \text{ if }\, (-\zeta_n^{-1})^{k+2} = 1.\\
\end{cases}\]
In the latter case, $-\zeta_n^{-1}$ is a simple root of $q$. 
\end{lemma}

\begin{proof}
Suppose that $\alpha$ is a common root of $q$ and $\q$. Then, 
    \[s_k(\alpha) =\zeta_n \quad \text{and}\quad  s_k(\alpha) = \zeta_n^{-1}\alpha^k.\]
Thus, $\alpha^k = \zeta_n^2$, implying $\alpha$ is a root of unity. Also, $\alpha\neq 1$ since $s_k(1)= k +1 \neq \zeta_n$. Write $\alpha =  e^{i\theta}$. See that 
\[ 1 = |\zeta_n|  = |s_k(\alpha)|= \left|\frac{\alpha^{k+1}-1}{\alpha-1}\right| = \left|\frac{\sin((k+1)\theta/2)}{\sin(\theta/2)}\right|,\]
where the last equality follows by the fact that $|e^{ix} - 1| = |2\sin\left(\frac{x}{2}\right)|$. 
It follows that  
\[ \sin^2\left(\frac{(k+1)\theta}{2}\right) = \sin^2\left(\frac{\theta}{2}\right),\]
and therefore either $k\theta \equiv 0\mod 2\pi$ or $(k+2)\theta \equiv 0 \mod 2\pi$. Equivalently, 
\[ \alpha^k = 1 \quad \text{or}\quad \alpha^{k+2}=1. \]
If $\alpha^k=1$, then $\zeta_n^2 =1$, which is a contradiction since $n>2$.

If $\alpha^{k+2}=1$, then 
\[\zeta_n  = s_k(\alpha) = \frac{\alpha^{k+1}-1}{\alpha-1} = \frac{\alpha^{-1}-1}{\alpha-1} = -\alpha^{-1},\]
implying $\alpha = -\zeta_n^{-1}$.

Conversely, it is easy to see by the above calculation that if we take $\alpha = -\zeta_n^{-1}$ and if $(-\zeta_n^{-1})^{k+2}=1$, then $\alpha$ is a root of $q$ and $\q$. 

To complete the proof, it suffices to show that if $(-\zeta_n^{-1})^{k+2}=1$, then $-\zeta_n^{-1}$ is not a double root of $q$. For $x\neq 1$, see that the derivative of $s_k(x)$ is given by 
\[ s_k'(x) = \frac{(k+1)x^k(x-1) - (x^{k+1}-1)}{(x-1)^2}.\]
Thus, 
\[s_k'(-\zeta_n^{-1}) = \frac{(-\zeta_n^{-1}+k+1)}{(-\zeta_n^{-1})^2(-\zeta_n^{-1}-1)},\]
using the fact that $(-\zeta_n^{-1})^{k+2}=1$. Clearly, $s'_k(-\zeta_n^
{-1})\neq 0$. Since the derivative of $q(x)$ is $q'(x) = s_k'(x)$, we have that $q'(-\zeta_n^{-1})\neq 0$, implying $-\zeta_n^{-1}$ is a simple root of $q$. This completes the proof. 
\end{proof}

The above lemma depends on whether $(-\zeta_n^{-1})^{k+2}$ equals to $1$ or not. It is easy to check for which values of $n$ and $k$ it equals $1$. 

\begin{lemma}\label{lemma:n_k_possibilities}
The equality $(-\zeta_n^{-1})^{k+2}=1$ holds if and only if one of the following holds:
\begin{enumerate}[$(i)$]
    \item $k$ is even and $n|(k+2)$;
    \item $k$ is odd, $n$ is even and $n|2(k+2)$. 
\end{enumerate}
\end{lemma}

Let 
\[g:=\gcd(q, \q) \quad \text{and}\quad r := \frac{q}{g}.\]
Next, we will show the following.

\begin{theorem}\label{thm:r_x_irreducible} For all $n>2$ and $k>1$, the polynomial $r(x)$ is irreducible over $\K$.
\end{theorem}

The irreducibility of $q$ can then be inferred using Lemma \ref{lemma:gcd}. We will first prove the above Theorem \ref{thm:r_x_irreducible} when $k\ge 3$. The case $k=2$ is a boundary case and needs to be dealt with separately.  

\subsection{Irreducibility of $r(x)$ over $\K$ when $k\ge 3$}
\label{subsec:k_atleast_3}

In this subsection, we show the following.

\begin{lemma}\label{lemma:k_atleast_3}
    If $k\ge 3$, then $r(x)$ is irreducible over $\K$.
\end{lemma}

\begin{proof}
Suppose, to the contrary, that 
\[r(x) = u(x)v(x),\]
where $u,v\in \K[x]$ are monic non-constant polynomials. It is clear that $u, v\in \mathcal{O}_{\K}[x]$ where $\mathcal{O}_{\K}$ is the ring of integers of $\K$. 

Let us quickly argue that the constant terms in $u$ and $v$ are non-zero so that $u^{\#}$ and $v^{\#}$ are well-defined. Since $q(0) = 1 - \zeta_n\neq 0$, and $g(0)\neq 0$ by Lemma \ref{lemma:gcd}, we conclude that $u(0)v(0)\neq 0$. 

Clearly, $q = g u v$. Define
\[ p:= guv^{\#}, \quad P:=(x-1)p,\quad Q := (x-1)q = x^{k+1} - \zeta_nx+\zeta_n-1.\]
It is clear that $p, P, Q\in \mathcal{O}_{\K}[x]$, the degree of $p$ is $k$, and $P, Q$ have degree $k+1$. Moreover, $p(0)$, $P(0)$ and $Q(0)$ are all non-zero. In particular, $p^{\#}, P^{\#}$ and $Q^{\#}$ are well-defined. 

Since $p^{\#} = g^{\#}u^{\#}v$, we see that 
\begin{equation}
   pp^{\#} = gg^{\#}uu^{\#}vv^{\#} = qq^{\#}. 
\end{equation}
Since $(x-1)^{\#} = 1-x$, we see that 
\begin{equation}
   PP^{\#} = QQ^{\#}. 
\end{equation}
Thus, $A_\ell(P) = A_\ell(Q)$. One can easily compute $A_\ell(Q)$ since $Q = x^{k+1} - \zeta_nx +\zeta_n-1$. Indeed, 
\begin{equation}\label{eq:autocorrelation_Q}
    A_\ell(P) = A_\ell(Q) = 
    \begin{cases}
    \zeta_n^{-1}-1 & \text{ if }\ell = k+1;\\
    -\zeta_n^{-1} & \text{ if }\ell = k;\\
    0 & \text{ if } 2\le \ell\le k-1 \\
    \zeta_n-1 & \text{ if }\ell = 1;\\
    4 - \zeta_n-\zeta_n^{-1} & \text{ if }\ell = 0.
    \end{cases}
\end{equation}

Write 
\[P(x) = \sum_{i=0}^{k+1} p_i x^i.\] 
In the next claim, we show that $P$ has small support. 

\begin{claim}\label{claim:small_support} The following hold.
\begin{enumerate}[$(i)$]
    \item If $\varphi(n)>2$, then $P$ has at most four non-zero coefficients, i.e., $|\supp(P)|\le 4$.
    \item If $\varphi(n) = 2$, then $P$ has at most three non-zero coefficients, i.e., $|\supp(P)|\le 3$.
\end{enumerate}
\end{claim}

\begin{proof} By \eqref{eq:autocorrelation_Q}, we have 
\[ A_{0}(P) = A_0(Q),\]
which is equivalent to 
\begin{equation}\label{eq:lag_0}
    \sum_{i=0}^{k+1} p_i\ol{p_i} =  4- \zeta_n-\zeta_n^{-1}.
\end{equation}
First suppose that $\varphi(n)>2$. Taking traces on both sides of \eqref{eq:lag_0} and using $\Tr_{\K/\Q}(\zeta_n)=\mu(n)$, we get
\begin{equation}\label{eq:trace_comparison}
   \sum_{i=0}^{k+1} \Tr_{\K/\Q}(p_i\ol{p_i}) = 4 \varphi(n)-2\mu(n). 
\end{equation}
Using \eqref{eq:trace_comparison} and Lemma \ref{lemma:trace_inequality}, we have 
\begin{equation}\label{eq:non_zero_coefficient}
   |\supp(P)| \,\varphi(n) \le 4\varphi(n)-2\mu(n). 
\end{equation}

If $|\supp(P)|\ge 5$, then 
\[ 5\varphi(n)\le 4\varphi(n)-2\mu(n)\le 4\varphi(n)+2 < 5\varphi(n),\]
a contradiction. Thus, $|\supp(P)|\leq 4$ in this case, proving assertion $(i)$.

Now, assume that $\varphi(n) = 2$. Then, $n\in \{3,4,6\}$. In these cases, $\K$ is an imaginary quadratic field and complex conjugation is a non-trivial element of $\Gal(\K/\Q)$. Therefore, for every element $\beta\in \mathcal{O}_{\K}$ the norm
\[\N_{\K/\Q}(\beta) = \beta\ol{\beta}\]
is a non-negative integer. Thus \eqref{eq:lag_0} is equivalent to 
\begin{equation}\label{eq:norm_case_2}
  \sum_{i=0}^{k+1} \N_{\K/\Q}(p_i) = \N_{\K/\Q}(\zeta_n-1) + 2. 
\end{equation}
Now,
\[ A_{k+1}(P)=A_{k+1}(Q)\]
is equivalent to
\[p_{k+1}\ol{p_0} = \zeta_n^{-1}-1 \]
by \eqref{eq:autocorrelation_Q}. Taking norms on both sides,  
\[ \N_{\K/\Q}(p_{k+1})\,\N_{\K/\Q}(p_0) = \N_{\K/\Q}(\zeta_n^{-1}-1).\]
Since $p_0, p_{k+1}, \zeta_n^{-1}-1$ are non-zero elements in $\mathcal{O}_{\K}$, $\N_{\K/\Q}(p_{k+1})$, $\N_{\K/\Q}(p_0)$ and $\N_{\K/\Q}(\zeta_n^{-1}-1)$ are positive integers. Moreover, since $n\in \{3,4,6\}$, $\N_{\K/\Q}(\zeta_n^{-1}-1)\in \{1,2,3\}$. In this case,
\begin{equation}\label{eq:norm_plus_1}
   \N_{\K/\Q}(p_{k+1}) + \N_{\K/\Q}(p_0)\ge \N_{\K/\Q}(\zeta_n^{-1}-1)+1. 
\end{equation}
Using \eqref{eq:norm_case_2} and \eqref{eq:norm_plus_1}, we see that there is at most one index $i\notin \{0,{k+1}\}$ such that $p_i\neq 0$. Therefore, $|\supp(P)|\le 3$, proving assertion $(ii)$. 
\end{proof}

We now determine the exact possibilities for $\supp(P)$. 
\begin{claim}\label{claim:support_exact} We have
    \[ \supp(P) =\{0,1,k+1\}\quad \text{or}\quad \supp(P) = \{0,k, k+1\}.\]
\end{claim}

\begin{proof}
By \eqref{eq:autocorrelation_Q}, 
\[0\neq A_{k+1}(P) = p_{k+1}\ol{p_0}\]
which implies $ 0, k+1\in \supp(P).$ Also, 
\[0\neq A_{k}(P) = p_k\ol{p_0} + p_{k+1}\ol{p_1}\]
which implies $1\in \supp(P)$ or $k\in \supp(P).$ Thus, either $\{0,1,k+1\}$ or $\{0,k,k+1\}$ is a subset of $\supp(P)$. 
Using Claim \ref{claim:small_support}, we have 
\[3\le |\supp(P)|\le 4.\]
If $|\supp(P)|=3$, then we are done. So suppose that $|\supp(P)|=4$. We consider the following cases:

\textbf{Case 1:} $k\ge 4$

Replacing $P$ with $P^{\#}$ if necessary, we can assume that $1\in \supp(P)$. So suppose
\[ \supp(P) = \{0,1,t, k+1\}\quad \text{where}\quad 2\le t\le k.\]
If $t=k$, then see that 
\[ A_{k-1}(P) = p_k\ol{p_1}\neq 0,\]
which contradicts \eqref{eq:autocorrelation_Q}.

If $2\le t\le k-1$ and $k+1\neq 2t$, then  
\[A_t(P) = p_t\ol{p_0}\neq 0,\]
a contradiction to \eqref{eq:autocorrelation_Q}.

If $2t=k+1$, then since $k\ge 4$, we have $t\ge 3$. In this case, 
\[ A_{t-1}(P) = p_t\ol{p_1}\neq 0,\]
again a contradiction to \eqref{eq:autocorrelation_Q}.

\textbf{Case 2:} $k = 3$

We need to rule out the possibilities $\{0,1,2,4\}$, $\{0,1,3,4\}$ and $\{0,2,3,4\}$ for $\supp(P)$. 

If $\supp(P)=\{0,1,3,4\}$, then 
\[A_2(P) = p_3\ol{p_1}\neq 0\]
which contradicts \eqref{eq:autocorrelation_Q}. 

Notice that if $\supp(P) = \{0,1,2,4\}$, then $\supp(P^{\#})=\{0,2,3,4\}$. Thus, to complete the proof, it suffices to rule out $\{0,1,2,4\}$. So write 
\[P(x) = p_0 + p_1x + p_2x^2 +p_4x^4.\]
Let 
\begin{equation}\label{eq:p_0}
   \alpha = \frac{p_0}{(\zeta_n-1)}\neq 0.
\end{equation}
Using \eqref{eq:autocorrelation_Q}, we have
\[    A_4(P)  = p_4\ol{p_0}  = \zeta_n^{-1}-1,\]
which implies
\begin{equation}\label{eq:p_4}
  p_4 = \frac{1}{\ol{\alpha}}.  
\end{equation}
Since
\[A_3(P) = p_4\ol{p_1} = - \zeta_n^{-1},\]
we get 
\begin{equation}\label{eq:p_1}
   p_1 = - \zeta_n \alpha. 
\end{equation}
As $P(1)=0$, we have 
\[ p_0+p_1 + p_2 +p_4 = 0,\]
which implies 
\begin{equation}\label{eq:p_2}
   p_2 = - p_0 - p_1-p_4 = \alpha - \frac{1}{\ol{\alpha}}. 
\end{equation}
Finally, using \eqref{eq:autocorrelation_Q}, 
\begin{equation}\label{eq:A_2}
    A_2(P) = p_2\ol{p_0} + p_4\ol{p_2} = 0.
\end{equation}
Using \eqref{eq:p_0}, \eqref{eq:p_1}, \eqref{eq:p_2} and \eqref{eq:p_4}, the above equation \eqref{eq:A_2} becomes  
\[ \left(\alpha\ol{\alpha}-1 \right)\left(\zeta_n^{-1}-1+\frac{1}{\alpha\ol{\alpha}}\right) = 0.\]
The second factor is non-real and therefore cannot vanish. Hence, $\alpha\ol{\alpha}=1$, which gives $p_2 = 0$ by \eqref{eq:p_2}. This is a contradiction since $2\in \supp(P)$. The proof is complete.
\end{proof}

\begin{claim}\label{claim:form_of_P}
We have 
    \[ P = \lambda Q\quad \text{or} \quad P = \lambda Q^{\#}\]
for some $\lambda\in \K$ such that $\lambda \ol{\lambda}=1$.     
\end{claim}

\begin{proof}
By Claim \ref{claim:support_exact}, we consider the following cases:

\textbf{Case 1:} $\supp(P) =\{0,1,k+1\}$.

Write
\[ P(x) = p_0 + p_1x+ p_{k+1}x^{k+1}.\]
Using \eqref{eq:autocorrelation_Q}, we have 
\[ A_{k+1}(P) = p_{k+1}\ol{p_0}=\zeta_n^{-1}-1\quad \text{ and }\quad A_{k}(P) = p_{k+1}\ol{p_1} = -\zeta_n^{-1}.\]
Thus, 
\begin{equation}\label{eq:p1_p0}
    p_1 = -\frac{\zeta_n}{\zeta_n-1}p_0.
\end{equation}
Since $P(1) =0$, we have $p_0 + p_1 + p_{k+1} = 0$. Using \eqref{eq:p1_p0}, we get 
\[ p_0 -\frac{\zeta_n}{\zeta_n-1}p_0 = -p_{k+1},\]
which implies
\[ p_0 = p_{k+1}(\zeta_n-1).\]
Similarly, 
\[p_1 = -p_{k+1}\zeta_n.\]
Substituting these expressions for $p_0$ and $p_1$ in $P(x)$, we see that 
\[ P(x) = p_{k+1}\left(x^{k+1}-\zeta_nx+\zeta_n-1  \right) = p_{k+1}Q(x).\]
Moreover, since $A_{k+1}(Q) = A_{k+1}(P) = A_{k+1}(p_{k+1}Q)$, we get $p_{k+1}\ol{p_{k+1}} = 1$. This completes the proof in this case. 

\textbf{Case 2:} $\supp(P) = \{0,k,k+1\}$

Then, $\supp(P^{\#})=\{0,1,k+1\}$. Arguing as in Case 1 with $P^{\#}$ gives the desired result.
\end{proof}

We are now ready to finish the proof of the lemma. We consider the following cases by Claim \ref{claim:form_of_P}:

\textbf{Case 1:} $P = \lambda Q$. ~ This means $(x-1)guv^{\#} = \lambda(x-1)guv$, which implies $v^{\#} = \lambda v$. Thus, $v$ divides both $q = guv$ and $q^{\#}=g^{\#}u^{\#}v^{\#}$, implying that $v$ divides their greatest common divisor $g$. This is a contradiction because $v$ is a factor of $r$ and $\gcd(r,g) =1$ by Lemma \ref{lemma:gcd}.  

\textbf{Case 2:} $P = \lambda Q^{\#}$. ~ One can argue that $u$ divides both $q$ and $q^{\#}$, which leads to the same contradiction as in Case 1. 
\end{proof}

\subsection{Irreducibility of $r(x)$ over $\K$ when $k = 2$}

In this case, 
\[ q(x) = 1 + x +x^2 - \zeta_n.\]
Using Lemmas \ref{lemma:gcd} and \ref{lemma:n_k_possibilities}, we see
\[ q(x) = 
\begin{cases}
    r(x) & \text{ if }n\neq 4;\\
    r(x)(x+\zeta_{4}^{-1}) & \text{ if }n=4.
\end{cases}\]

We show the following. 

\begin{lemma}\label{lemma:k_equal_2}
    If $k=2$, then $r(x)$ is irreducible over $\K$.
\end{lemma}

\begin{proof}
First, assume $n=4$. Take $\zeta_4 = i$. Then by the above,
\[ r(x) = x + 1 + i,\]
and therefore $r(x)$ is irreducible.

So, assume that $n\neq 4$. Then $r(x) = q(x)$, gcd $g(x)=1$, and degree of $r$ is $2$. We will proceed similarly to the proof for $k\ge 3$ in Lemma \ref{lemma:k_atleast_3} given in Subsection \ref{subsec:k_atleast_3}. Suppose, to the contrary, that 
\[ r(x) = u(x)v(x),\]
where $u(x)$ and $v(x)$ are monic and linear. Let $p, P, Q$ be as before. Write 
\[P(x) = p_0 + p_1x + p_2x^2 + p_3x^3.\]

\begin{claim}\label{claim:full_support_P}
    We have $\supp(P) = \{0,1,2,3\}.$
\end{claim}

\begin{proof}
By \eqref{eq:autocorrelation_Q}, we have 
\begin{equation}\label{eq:lag_3_k_2}
    0\neq \zeta_n^{-1}-1 = A_3(P) =p_3\ol{p_0} 
\end{equation}
which implies $0,3\in \supp(P).$ Again, by \eqref{eq:autocorrelation_Q},
\begin{equation}
    0\neq  -\zeta_n^{-1} = A_2(P) =p_2\ol{p_0} + p_3\ol{p_1}
\end{equation}
which implies $ 1\in \supp(P) \text{ or }2\in \supp(P).$ Thus, $|\supp(P)|\ge 3$. Suppose, to the contrary, that $|\supp(P)| = 3$. We have the following possibilities:

\textbf{Case 1:} $\supp(P) = \{0,1,3\}$

Solving the equations 
\[  \zeta_n^{-1}-1 = p_3\ol{p_0}, \qquad -\zeta_n^{-1} = p_3\ol{p_1},\qquad P(1) = p_0 +p_1+p_3 = 0,\]
gives 
\[ p_0 = p_3(\zeta_n-1), \qquad p_1 = - p_3\zeta_n,\qquad p_3\ol{p_3}=1.\]
It follows that 
\[P = p_3 Q \quad \text{with}\quad p_3\ol{p_3}=1.\] 
This means $(x-1)uv^{\#} = p_3(x-1)uv$ implying $v^{\#} = p_3v$. Thus, $v$ divides both $q$ and $q^{\#}$, which gives $\gcd(q,q^{\#})\neq 1$, a contradiction.  

\textbf{Case 2:} $\supp(P) = \{0,2,3\}$

In this case, $\supp(P^{\#}) = \{0,1,3\}$. Arguing as in Case 1, we get a similar contradiction.
\end{proof}

Now, as in the proof of Claim \ref{claim:small_support} (see \eqref{eq:trace_comparison}), it holds that 
\begin{equation}\label{eq:trace_P_k_2}
   \sum_{i=0}^{3} \Tr_{\K/\Q}(p_i\ol{p_i}) = 4 \varphi(n)-2\mu(n). 
\end{equation}
By Claim \ref{claim:full_support_P}, $p_0, p_1, p_2, p_3$ are all non-zero. By Lemma \ref{lemma:trace_inequality}, we have 
\begin{equation}\label{eq:p_1_p_2_k_2}
    \Tr_{\K/\Q}(p_1\ol{p_1}) + \Tr_{\K/\Q}(p_2\ol{p_2})\ge 2\varphi(n).
\end{equation}
Moreover, for any embedding $\sigma:\K \rightarrow \C$,
\[ |\sigma(p_0)||\sigma(p_3)| = |\sigma(\zeta_n)-1|\]
by \eqref{eq:lag_3_k_2}. Using the AM-GM inequality, we get 
\[ |\sigma(p_0)|^2 + |\sigma(p_3)|^2 \ge 2|\sigma(\zeta_n)-1|.\]
Summing over all embeddings, 
\begin{align}\label{eq:p_0_p_3_k_2}
   \Tr_{\K/\Q}(p_0\ol{p_0}) +  \Tr_{\K/\Q}(p_3\ol{p_3})& \ge 2 \sum_{\substack{1\le j\le n\\\gcd(j,n)=1}} |\zeta_n^j-1|\nonumber\\ 
   & > \sum_{\substack{1\le j\le n\\\gcd(j,n)=1}} |\zeta_n^j-1|^2\nonumber\\
   & = \sum_{\substack{1\le j\le n\\\gcd(j,n)=1}} (2-\zeta_n^j-\zeta_n^{-j})\nonumber\\
   & = 2\varphi(n)-2\mu(n).
\end{align}
Here, the strict inequality holds because $|\zeta_n^j-1|<2$, and so $|\zeta_n^j-1|^2 < 2|\zeta_n^j-1|$. Combining \eqref{eq:p_1_p_2_k_2} and \eqref{eq:p_0_p_3_k_2}, we get 
\[ \sum_{i=0}^{3} \Tr_{\K/\Q}(p_i\ol{p_i}) > 4 \varphi(n)-2\mu(n),\]
which contradicts \eqref{eq:trace_P_k_2}. The proof of the lemma is complete.
\end{proof}

Combining Lemmas \ref{lemma:k_atleast_3} and \ref{lemma:k_equal_2} gives Theorem \ref{thm:r_x_irreducible}.

\subsection{Final step}

We are ready to finish the proof of Theorem \ref{thm:cyclotomic}. 

\begin{proof}[Proof of Theorem \ref{thm:cyclotomic}]
Observe that
\[F_{n,k}(x) = \prod_{\sigma\in \Gal(\K/\Q)}(s_k(x)-\sigma(\zeta_n)) =  \prod_{\sigma\in \Gal(\K/\Q)}\sigma(q(x)).\]
We consider the following cases:

\textbf{Case 1:} $k\ge 3$

If $(-\zeta_n^{-1})^{k+2}\ne 1$, then $g=1$ by Lemma \ref{lemma:gcd}. Thus, $q = r$, and by Theorem \ref{thm:r_x_irreducible} $q$ is irreducible over $\K$. Let $\alpha$ be a root of $q$. Then 
\[ \zeta_n = s_k(\alpha)\in \Q(\alpha),\]
which implies $\K\subseteq \Q(\alpha).$ Hence, the degree of extension
\[[\Q(\alpha):\Q] = [\Q(\alpha):\K][\K:\Q] = k \varphi(n).\]
But $\alpha$ is a zero of $F_{n,k}$ and degree of $F_{n,k}$ is also $k\varphi(n)$. Therefore, $F_{n,k}$ is the minimal polynomial of $\alpha$ over $\Q$, and so $F_{n,k}$ is irreducible.

Now, assume that $(-\zeta_n^{-1})^{k+2}= 1$. Then, by Lemma \ref{lemma:gcd}, 
\[q(x) = (x+\zeta_n^{-1})r(x),\]
where $r(x)$ is irreducible over $\K$ by Theorem \ref{thm:r_x_irreducible}.

Let $M$ denote the order of $-\zeta_n^{-1}$, i.e., $M$ is the smallest positive integer such that $(-\zeta_n^{-1})^M = 1$. Since $\Q(-\zeta_n^{-1}) =\K$, we have $\varphi(M) = \varphi(n)$ and
\[F_{n,k}(x) = \prod_{\sigma\in \Gal(\K/\Q)}\sigma(r(x))\prod_{\sigma\in \Gal(\K/\Q)}(x + \sigma(\zeta_n^{-1})) = H(x)\ \Phi_M(x),\]
where 
\[H(x) := \prod_{\sigma\in \Gal(\K/\Q)}\sigma(r(x))\in \Q[x].\]
Since cyclotomic polynomials are always irreducible, $\Phi_M(x)$ is irreducible. We will argue that $H(x)$ is irreducible. 

Consider a root $\beta$ of $r(x)$. Then 
\[\zeta_n = s_k(\beta)\in \Q(\beta),\]
implying $\K\subseteq \Q(\beta)$. Hence, 
\[[\Q(\beta):\Q] = [\Q(\beta):\K][\K:\Q] = (k-1) \varphi(n).\]
Here, $[\Q(\beta):\K] = k-1$ because  $r$ is irreducible over $\K$ and has degree $k-1$. Since $\beta$ is a root of $H(x)$ and degree of $H(x)$ is $(k-1)\varphi(n)$, $H(x)$ is precisely the minimal polynomial of $\beta$ over $\Q$. It follows that $H(x)$ must be irreducible. 

Using Lemma \ref{lemma:n_k_possibilities}, the proof of Theorem \ref{thm:cyclotomic} is complete when $k\ge 3$.

\textbf{Case 2:} $k = 2$

If $n\neq 4$, then $q = r$, and therefore $q$ is irreducible over $\K$  by Lemma \ref{lemma:k_equal_2}. A similar degree of extension argument as in Case 1 shows that $F_{n,2}$ is irreducible over $\Q$. 

Finally, suppose $n=4$. Take $\zeta_4 = i$. Then 
\[ F_{4,2}(x) = (1 +x+x^2-i)(1 +x+x^2+i) = (x^2 + 1)(x^2+2x+2),\]
and both quadratic factors on the right-hand side are clearly irreducible over $\Q$. This completes the proof of Theorem \ref{thm:cyclotomic} when $k=2$.  
\end{proof}

\section*{Acknowledgements}
The authors thank Gaurish Korpal for helpful comments. 

\section*{AI statement}

We acknowledge the use of AI tools during the ideation phase. We declare that the text is not AI-generated.

\bibliographystyle{plain}
\bibliography{references}

\vspace{0.4cm}

\affl{Jaskaran Kaur}{jkj10@sfu.ca}{Department of Mathematics, Simon Fraser University, Burnaby, Canada}

\affl{Hitesh Kumar}{hitesh.kumar.math@gmail.com, hitesh\_kumar@sfu.ca}{Department of Mathematics, Simon Fraser University, Burnaby, Canada}
\end{document}